\documentclass[11pt, letterpaper]{amsart}
\usepackage[utf8]{inputenc}
\usepackage{setspace}
\usepackage{mathrsfs}
\usepackage{amssymb}
\usepackage{latexsym}
\usepackage{amsfonts}
\usepackage{amsmath}
\usepackage{eucal}
\usepackage{bm}
\usepackage{bbm}
\usepackage{graphicx}
\usepackage[english]{varioref}
\usepackage[nice]{nicefrac}
\usepackage[all]{xy}
\usepackage{amsthm}
\usepackage{amssymb,amsthm,upref,amscd}

\def\m{\mathbb{M}}
\def\op{\operatorname}

\def\mmod{\kern-1pt\operatorname{-mod}}

\def\bg{{\bf G}}

\def\bk{\Bbbk}

\def\la{\lambda}
\def\bc{{\bf C}}
\def\f{\mathbb{F}}

\newtheorem{theorem}{Theorem}[section]
\newtheorem{lemma}[theorem]{Lemma}

\newtheorem{remark}[theorem]{Remark}
\newtheorem{proposition}[theorem]{Proposition}

\theoremstyle{proposition}

\numberwithin{equation}{section}

\begin{document}

\title[Block decomposition of principal representation category]{The block decomposition of the principal representation category of reductive algebraic groups with Frobenius maps}

\author{Xiaoyu Chen}
\address{Department of Mathematics, Shanghai Normal University,
100 Guilin Road, Shanghai 200234, PR China.}
\email{chenxiaoyu@shnu.edu.cn}

\author{Junbin Dong}
\address{Institute of Mathematical Sciences, ShanghaiTech University, 393 Middle Huaxia Road, Pudong, Shanghai 201210, PR China.}
\email{dongjunbin@shanghaitech.edu.cn}

\subjclass[2010]{20C07, 20G05}

\date{}

\keywords{Algebraic group,  principal representation, extension}

\begin{abstract}
Let ${\bf G}$ be a connected reductive algebraic group defined over the finite field $\mathbb{F}_q$ with $q$ elements. Let $\Bbbk$ be a field such that  $\op{char} \Bbbk \ne \op{char} \mathbb{F}_q$.  In this paper, we study the extensions of simple modules (over $\Bbbk$) in the  principal representation category $\mathscr{O}(\bf G)$ which is defined in \cite{D1}. In particular, we get the block decomposition of   $\mathscr{O}(\bf G)$, which is parameterized by the central  characters  of ${\bf G}$.
\end{abstract}

\maketitle

\section{Introduction}
Let ${\bf G}$ be a connected reductive algebraic  group defined over the finite field $\mathbb{F}_q$ with $q$ elements.  All the representations we considered in this paper are over a field $\Bbbk$, where $\Bbbk$ satisfies $\op{char} \Bbbk \ne \op{char} \mathbb{F}_q$. Let ${\bf B}$ be a Borel subgroup of {\bf G}. Let $\bf T$ be a maximal  torus contained in the  Borel subgroup $\bf B$ and $\theta$  be   a character of ${\bf T}$. Thus $\theta$  can also be regarded as a character of ${\bf B}$ by letting ${\bf U}$ (the unipotent radical of ${\bf B}$) act trivially. Let $\Bbbk_\theta$ be the one-dimensional representation of ${\bf B}$ affording $\theta$. The naive induced module $\mathbb{M}(\theta)=\Bbbk{\bf G}\otimes_{\Bbbk{\bf B}}{\Bbbk}_\theta$   was deeply studied in \cite{CD1}.  In particular,  $\mathbb{M}(\theta)$  has a composition series of finite length and the composition factors of $\mathbb{M}(\theta)$  are $E(\theta)_J$ with $J\subset I(\theta)$ (see Section 2 for the explicit definition of $I(\theta)$ and $E(\theta)_J$). Not long after, J.B. Dong introduced the principal representation category $\mathscr{O}({\bf G})$  in  \cite{D1}, which  was conjectured to be a highest weight category in the sense of \cite{CPS} when $\Bbbk =\mathbb{C}$.  In particular, the extensions of the simple modules in $\mathscr{O}({\bf G})$ may have good properties.  But soon after,  X.Y. Chen constructed a counter example (see \cite{C3}) to show that this conjecture is not true.  Thus we are interested  in the extensions of the representations in   $\mathscr{O}(\bg)$ and in \cite{CD2} we showed that  for any two characters $\lambda,\mu$ of $\bf T$,  $\op{Ext}_{\bk\bg}^1(\m(\lambda),\m(\mu))=0$ if and only if $\lambda|_{\bf C}\ne\mu|_{\bf C}$, where $\bc$ is the center of $\bg$.  In this paper, we study the extensions between simple modules  in the category $\mathscr{O}(\bg)$.  The main theorem is as follows:

\begin{theorem} \label{mainthm} \normalfont
Let $\lambda,\mu$ be two characters of $\bf T$.  If  $\lambda|_{\bf C}\ne\mu|_{\bf C}$, then $\op{Ext}_{\bk\bg}^1(E(\lambda)_J, E(\mu)_K)=0$. If $\lambda|_{\bf C}= \mu|_{\bf C}$  and $ E(\mu)_K$ is infinite dimensional, then $\op{Ext}_{\bk\bg}^1(E(\lambda)_J, E(\mu)_K)\ne 0$.
\end{theorem}

 This paper is organized as follows:  In Section 2, we give some notations and preliminary results. In Section 3, we prove the main  theorem and thus give the block decompositions of the category $\mathscr{O}(\bg)$,  which is parameterized by the central  characters  of ${\bf G}$.

\section{Preliminaries}
As in the introduction,  ${\bf G}$ is a connected reductive algebraic group defined over $\mathbb{F}_q$ with the standard Frobenius homomorphism $\text{Fr}$ induced by the automorphism $x\mapsto x^q$ on $\bar{\mathbb{F}}_q$, where $q$ is a power of a prime number $p$.  Denote by ${\bf C}$ the center of ${\bf G}$.  Let ${\bf B}$ be an $\text{Fr}$-stable Borel subgroup, and ${\bf T}$ be an $\text{Fr}$-stable maximal torus contained in ${\bf B}$, and ${\bf U}=R_u({\bf B})$ be the ($\text{Fr}$-stable) unipotent radical of ${\bf B}$. We identify ${\bf G}$ with ${\bf G}(\bar{\mathbb{F}}_q)$ and do likewise for the various subgroups of ${\bf G}$ such as ${\bf B}, {\bf T}, {\bf U}$ $\cdots$. We denote by $\Phi=\Phi({\bf G};{\bf T})$ the corresponding root system, and by $\Phi^+$ (resp. $\Phi^-$) the set of positive (resp. negative) roots determined by ${\bf B}$. Let $W=N_{\bf G}({\bf T})/{\bf T}$ be the corresponding Weyl group and denote by $\ell(w)$ the length of $w\in W$. We denote by $\Delta=\{\alpha_i\mid i\in I\}$ the set of simple roots, and by $S=\{s_i\mid i\in I\}$ the corresponding simple reflections. For each $\alpha\in\Phi$, let ${\bf U}_\alpha$ be the root subgroup corresponding to $\alpha$ and we fix an isomorphism $\varepsilon_\alpha: \bar{\mathbb{F}}_q\rightarrow{\bf U}_\alpha$ such that $t\varepsilon_\alpha(c)t^{-1}=\varepsilon_\alpha(\alpha(t)c)$ for any $t\in{\bf T}$ and $c\in\bar{\mathbb{F}}_q$. For any $w\in W$, let $\Phi_w^+=\{\alpha\in\Phi^+\mid w(\alpha)\in\Phi^+\}$ and $\Phi_w^-=\{\alpha\in\Phi^+\mid w(\alpha)\in\Phi^-\}$.
Let ${\bf U}_w$ (resp.  ${\bf U}'_w$ ) be the subgroup of ${\bf U}$ generated by all ${\bf U}_\alpha$  with $\alpha\in\Phi_w^-$  (resp.  $\alpha\in\Phi_w^+$ ).  One can refer to \cite{Car} for  the structure theory of algebraic groups.

For each positive integer $k$, we denote $G_k$ for the $\f_{q^{k!}}$-points of $\bg$, and do likewise for $B_k,T_k,U_k,U_{w,k}$ etc.  Now let  ${\bf H}$ be a group which has a sequence of finite subgroups $H_1, H_2, \dots, H_k, \dots $ such that  ${\bf H}=\displaystyle  \bigcup_{i=1}^{\infty} H_i$ and $H_i \subset H_j$ whenever $i<j$. Clearly, the groups ${\bf G}, {\bf B}, {\bf T}, {\bf U}$ satisfy this property. Let $\{M_i, \varphi_{ij}\}$ be a direct system of vector spaces over $\Bbbk$.  Assume that  $M_i$ is a $\Bbbk H_i$-module for each integer $i$ and the morphisms $\varphi_{ij}:$ $M_i\rightarrow M_j$ satisfy $\varphi_{ij}(hm)=h\varphi_{ij}(m)$ for any $h\in H_i$ and $m\in M_i$ whenever $i<j$.  Then the direct limit $M:=\varinjlim M_i$ is a $\Bbbk {\bf H}$-module by \cite[Lemma 1.5]{Xi}. Namely, let $h\in{\bf H}$ and $m\in M$, then $h\in H_i$ for some $i$ and $m_j\in M_j$ represents $m$ for some $j$, choose an integer  $k>i,j$ and define $hm$ as the element of $M$ represented by $h\varphi_{jk}(m_j)\in M_k$.

This is a useful method to construct the abstract representations of algebraic groups. Moreover, N.H. Xi showed that if each $M_i$ is an irreducible $\Bbbk H_i$-module, then the direct limit $\varinjlim M_i$ is an irreducible  $\Bbbk {\bf H}$-module.

 Let $\widehat{\bf T}$ be the set of characters of ${\bf T}$ over ${\Bbbk}$. Each $\theta\in\widehat{\bf T}$ is regarded as a character of ${\bf B}$ by the homomorphism ${\bf B}\rightarrow{\bf T}$. Let ${\Bbbk}_\theta$ be the corresponding ${\bf B}$-module. Let $\mathbb{M}(\theta)=\Bbbk{\bf G}\otimes_{\Bbbk{\bf B}}{\Bbbk}_\theta$ and we write $x{\bf 1}_{\theta}:=x\otimes{\bf 1}_{\theta}\in \mathbb{M}(\theta)$ for short, where ${\bf 1}_{\theta}$  is a nonzero element in ${\Bbbk}_\theta$. Using the Bruhat decomposition, we see that $$\mathbb{M}(\theta)= \sum_{w\in W}\Bbbk {\bf U}_{w^{-1}} \dot{w} {\bf 1}_{\theta}.$$

For each $i \in I$, let ${\bf G}_i$ be the subgroup of $\bf G$ generated by ${\bf U}_{\alpha_i}, {\bf U}_{-\alpha_i}$ and we set ${\bf T}_i= {\bf T}\cap {\bf G}_i$. For $\theta\in\widehat{\bf T}$, we let $$I(\theta)=\{i\in I \mid \theta| _{{\bf T}_i} \ \text {is trivial}\}.$$
For $J\subset I(\theta)$, let $W_J$ be the subgroup of $W$ generated by $s_i~(i\in J)$ and $w_J$ be the longest element of $W_J$.  Let ${\bf G}_J$ be the subgroup of $\bf G$ generated by ${\bf G}_i~(i\in J)$. We choose a representative $\dot{w}\in {\bf G}_J$ for each $w\in W_J$. Thus, the element $w{\bf 1}_\theta:=\dot{w}{\bf 1}_\theta$  $(w\in W_J)$ is well defined. For $J\subset I(\theta)$, we set
$$\eta(\theta)_J=\sum_{w\in W_J}(-1)^{\ell(w)}w{\bf 1}_{\theta},$$
and  let $\mathbb{M}(\theta)_J=\Bbbk{\bf G}\eta(\theta)_J$ be the $\Bbbk {\bf G}$-module which is generated by $\eta(\theta)_J$.

The structure of  $\mathbb{M}(\theta)_J$ is very useful to our later discussion.  We state the results here. One can refer to \cite[Section 2]{CD1} for more details. Denote by $$W^J =\{x\in W\mid x~\op{has~minimal~length~in}~xW_J\}$$ and we have
$$\mathbb{M}(\theta)_J=\sum_{w\in W^J}\Bbbk {\bf U}_{w_Jw^{-1}}\dot{w}\eta(\theta)_J.$$
Now we give the formula of $\dot{s_i}\varepsilon_i(x)\dot{w} \eta(\theta)_J$ for $w\in W^J$, $x\in\bar{\mathbb{F}}_q^*$, and $\alpha_i\in\Phi_{w_Jw^{-1}}^-$. The details are referred to \cite[Proposition 2.5]{CD1}.
For each $x\in\bar{\mathbb{F}}_q^*$, we have
$$\dot{s_i}\varepsilon_i(x)\dot{s_i}^{-1}=f_i(x)\dot{s_i}h_i(x)g_i(x),$$
where $f_i(x),g_i(x) \in {\bf U}_{\alpha_i}\backslash\{e\}$, and $h_i(x)\in {\bf T}_i$ are uniquely determined. Let $\le$ be the Bruhat order in $W$.
Noting that  $w_Jw^{-1}(\alpha_i)\in\Phi^-$,  we get $s_i w w_J \leq ww_J$. By easy calculation, we have the following

\begin{proposition} \label{suwformula}
With the notations above,

\noindent (i) If $s_iw \leq w$, then
$$\dot{s_i} \varepsilon_i(x) \dot{w} \eta(\theta)_J=\theta^w(\dot{s_i}h_i(x)\dot{s_i})f_i(x)\dot{w}\eta(\theta)_J.$$

\noindent  (ii) If $w \leq  s_iw$ but $s_iww_J\leq ww_J$, then
$$\dot{s_i} \varepsilon_i(x)  \dot{w} \eta(\theta)_J =(f_i(x)-1)\dot{w}\eta(\theta)_J.$$
\end{proposition}

We define
$$E(\theta)_J=\mathbb{M}(\theta)_J/N(\theta)_J,$$
where $N(\theta)_J$ is the sum of all $\mathbb{M}(\theta)_K$ with $J\subsetneq K\subset I(\theta)$. We denote by $C(\theta)_J$ the image of $\eta(\theta)_J$ in $E(\theta)_J$.
Set
$$ Z_J(\theta)  =\{w\in X_J \mid \mathscr{R}(ww_J)\subset J\cup (I\backslash I(\theta))\},$$
where $\mathscr{R}(w)= \{s\in S \mid ws <w\}$.
\begin{proposition} \cite[Proposition 2.7]{CD1} \label{DesEJ}
For $J\subset I(\theta)$, we have
$$E(\theta)_J=\sum_{w\in Z_J(\theta)}\Bbbk {\bf U}_{w_Jw^{-1}}\dot{w}C(\theta)_J,$$
and  the set $\{u\dot{w}C(\theta)_J \mid w\in Z_J(\theta), u\in {\bf U}_{w_Jw^{-1}} \}$ is a basis of $E(\theta)_J$.
\end{proposition}

By  \cite[Theorem 3.1]{CD1},  the composition factors of  $\mathbb{M}(\theta)$  are $E(\theta)_J$ $ (J\subset I(\theta))$,  with each of multiplicity one. According to \cite[Proposition 2.8]{CD1}, one has that  $E(\theta_1)_{K_1}$ is isomorphic to $E(\theta_2)_{K_2}$ as $\Bbbk {\bf G}$-modules if and only if $\theta_1=\theta_2$ and $K_1=K_2$.

The irreducible $\Bbbk {\bf G}$-modules $E(\theta)_J$ can also be realized by parabolic induction.
For any $J\subset I$, let ${\bf P}_J$ be the corresponding standard parabolic subgroup of  ${\bf G}$.
We  have the Levi decomposition ${\bf P}_J={\bf L}_J\ltimes{\bf U}_J$, where ${\bf L}_J$ is the subgroup of ${\bf P}_J$ generated by ${\bf T}$, and all ${\bf U}_{\alpha_i}$ and ${\bf U}_{-\alpha_i}$ with $i\in J$ and ${\bf U}_J=R_u({\bf P}_J)$. Let $\theta\in\widehat{\bf T}$ and $K\subset I(\theta)$. Since $\theta|_{{\bf T}_i}$ is trivial for all $i\in K$, it induces a character (still denoted by $\theta$) of $\overline{\bf T}={\bf T}/{\bf T}\cap[{\bf L}_K,{\bf L}_K]$. Therefore, $\theta$ is regarded as a character of ${\bf L}_K$ by the homomorphism ${\bf L}_K\rightarrow\overline{\bf T}$, and hence as a character of ${\bf P}_K$ by letting ${\bf U}_K$ acts trivially. Let $\Bbbk_\theta$ be the one-dimensional representation of ${\bf P}_K$ affording $\theta$. We set $\mathbb{M}(\theta, K):=\Bbbk{\bf G}\otimes_{\Bbbk{\bf P}_K}\theta$.  Let ${\bf 1}_{\theta, K}$ be a nonzero element in the one-dimensional module $\Bbbk_\theta$ associated to $\theta$. We abbreviate $x{\bf 1}_{\theta, K}:=x\otimes{\bf 1}_{\theta, K} \in \mathbb{M}(\theta, K)$ as before.
For $J\subset I(\theta)$, set $J'=I(\theta)\backslash J$ and  we denote by $\nabla(\theta)_J= \mathbb{M}(\theta, J')=\Bbbk{\bf G}\otimes_{\Bbbk{\bf P}_{J'}}\Bbbk_\theta$.  Let $E(\theta)_J'$ be the submodule of $\nabla(\theta)_J$ generated by $$D(\theta)_J:=\sum_{w\in W_J}(-1)^{\ell(w)}\dot{w}{\bf 1}_{\theta, J'}.$$
We see that $E(\theta)_J'$ is isomorphic to $E(\theta)_J$ as $\Bbbk {\bf G}$-modules by \cite[Proposition 1.9]{CD1}.  In the following, we will not make a distinction between $E(\theta)_J$ and $E(\theta)_J'$,  and regard $E(\theta)_J$ as the socle of $\nabla(\theta)_J$.

\section{The extensions of the simple modules in $\mathscr{O}({\bf G})$}

Let $\lambda, \mu$ be two characters of ${\bf T}$.  In this section, we study the extensions $\op{Ext}_{\bk\bg}^1(E(\lambda)_J, E(\mu)_K)$, where $J\subset I(\lambda)$ and $K\subset I(\mu)$.

\begin{theorem} \label{ifpart}
If $\lambda|_{\bf C}\ne\mu|_{\bf C}$, then $\op{Ext}_{\bk\bg}^1(E(\lambda)_J, E(\mu)_K)=0$.
\end{theorem}

\begin{proof} Let
$$0\rightarrow E(\mu)_K \rightarrow M \rightarrow E(\lambda)_J \rightarrow0$$
be a short exact sequence of $\Bbbk{\bf G}$-modules. Noting that  $$E(\lambda)_J=\sum_{w\in Z_J(\lambda)}\Bbbk {\bf U}_{w_Jw^{-1}}\dot{w}C(\lambda)_J,$$
let $\xi(\lambda)_J\in M$  such that its image is  $C(\lambda)_J$.  Since $\lambda|_{\bf C}\ne\mu|_{\bf C}$,  there is a $c_0\in\mathbb{\bf C}$ such that $\lambda(c_0)\ne\mu(c_0)$. We have
$$c_0 \xi(\lambda)_J =\lambda(c_0) \xi(\lambda)_J + m_0$$ for some $m_0 \in E(\mu)_K$. It is easy to see that $c_0 m=\mu(c_0) m$ for any $m\in E(\mu)_K$. Let $v_0=\xi(\lambda)_J  +(\la(c_0)-\mu(c_0))^{-1}m_0\in M $. Then it is easy to see that
$$ c_0v_0=\lambda(c_0) \xi(\lambda)_J + m_0  +\mu(c_0)  (\la(c_0)-\mu(c_0))^{-1}m_0 =\lambda(c_0)v_0.$$
Now we set $M_{\lambda(c_0)} =\{v\in M \mid c_0v= \lambda(c_0)v \}$, which is nonzero since $v_0\in M_{\lambda(c_0)}$.
We also have $M_{\lambda(c_0)} \cap  E(\mu)_K =0$. It is clear that $M_{\lambda(c_0)}$ is a $\Bbbk {\bf G}$-module since $c_0$ commutes with each element in ${\bf G}$. Noting that  $E(\lambda)_J, E(\mu)_K$ are simple  $\Bbbk {\bf G}$-modules, we get $M_{\lambda(c_0)} \cong E(\lambda)_J$ as $\Bbbk {\bf G}$-modules. The theorem is proved.
\end{proof}

\begin{theorem}\label{onlyifpart}
If $\lambda|_{\bf C}= \mu|_{\bf C}$  and $E(\mu)_K$ is infinite dimensional, then $\op{Ext}_{\bk\bg}^1(E(\lambda)_J, E(\mu)_K)\ne 0$.
\end{theorem}

\begin{remark} \normalfont
If $\text{char}\ \Bbbk \geq 5$  and $E(\mu)_K$ is finite dimensional,  we always have $\op{Ext}_{\bk\bg}^1(E(\lambda)_J, E(\mu)_K)=0$  by \cite[Theorem3.1]{CD2}.
\end{remark}

To prove Theorem \ref{onlyifpart}, we will  construct a non-split  exact sequence
\begin{equation}\label{exactseq}
0 \rightarrow E(\mu)_K  \rightarrow M \rightarrow E(\lambda)_J \rightarrow 0.
\end{equation}
For any integer $i$, denote $E_i(\lambda)_J= \Bbbk G_i C(\lambda)_J$ and  let $M_i= E_i (\lambda)_J \oplus  E_i(\mu)_K$, which is a $\Bbbk G_i$-module.  Let $J'=I(\lambda)\backslash J$. The character $\lambda$ can be regarded as a character of $P_{J'}$. For each $i$ and  $u\in U_{i+1}$,  we set
$$\xi_{i, u}=(\sum_{w\in W_J} (-1)^{\ell(w)}w) \sum_{p\in P_{J', i}} \lambda(p)^{-1}pu  w_0 C(\mu)_K \in E_{i+1}(\mu)_K.$$
The key result used to prove the non-splitness of \eqref{exactseq} is the following

\begin{proposition} \label{keylemma}
For sufficiently large integer $i$, there exists $u\in U_{i+1}$ such that $\xi_{i,u}$ is nonzero.
\end{proposition}
To prove this proposition, we need some preliminaries. Let $w_0$ be the longest element of $W$. For any $w\in W$, define
$$\Omega_w=\{x\in U_{i+1}\mid\dot{w} x\dot{w_0}  C(\mu)_K\in \Bbbk U_{i+1}w_0C(\mu)_K\}.$$
Set $ \Omega'_w= U_{w, i+1} \cap \Omega^c_w$, where $\Omega^c_w$ is the complementary set of $\Omega_w$ in $U_{i+1}$. For any $u\in U_{i+1}$, we denote $H_u=  \bigcup_{t\in T_i}U_i t u t^{-1}$,  which is a subset of $U_{i+1}$.
Now we set
$$ \Omega =\{u\in U_{i+1} \mid H_u \subset \bigcap_{w\in W} \Omega_w \}.$$
Thus if $u\in  \Omega$, then we have $\xi_{i,u} \in  \Bbbk U_{i+1}w_0C(\mu)_K$.
By the setting of  $\Omega$, it is easy to see that $ \Omega \cap U_i =\emptyset$.
We abbreviate $n=|\Phi^+|$ and $r=|I|$ in the sequel. We need the following technical Lemmas to prove Proposition \ref{keylemma}.

\begin{lemma}\label{owow'}
For any $w\in W$,  one has that $|\Omega'_w| \leq  (3\widetilde{q})^{\ell(w)-1}$ and $|\Omega_w|\ge\widetilde{q}^{n-1} ( \widetilde{q}-3^{\ell(w)-1})$, where  $\widetilde{q}= q^{(i+1)!}$.
\end{lemma}
\begin{proof}
 Noting that $U_{i+1}=U'_{w,i+1}U_{w,i+1}$, and hence for any $x\in\Omega_w$, $x=y'y$ for unique $y'\in U_{w,i+1}'$ and $y\in U_{w,i+1}$. Thus we have $$\dot{w}x\dot{w_0}C(\mu)_K=\dot{w}y'\dot{w}^{-1}\dot{w}y\dot{w_0}  C(\mu)_K\in\Bbbk U_{i+1}w_0C(\mu)_K.$$
Since $\dot{w}y'\dot{w}^{-1}\in U_{i+1}$, we get $\dot{w}y\dot{w_0}C(\mu)_K\in\Bbbk U_{i+1}w_0C(\mu)_K$   and hence $y\in\Omega_w\cap U_{w,i+1}$.
It follows that
\begin{equation}\label{omegaw}
|\Omega_w|=|U'_{w,i+1}||U_{w,i+1}\backslash\Omega'_w|=\widetilde{q}^{n-\ell(w)} (\widetilde{q}^{\ell(w)}-|\Omega'_w|)
\end{equation}
We will show the first statement by induction on $\ell(w)$.  It is easy to see that $|\Omega'_s|= 1$ for any simple reflection $s\in S$.
Now let $\ell(w) \geq 1$ and  $s\in  \mathscr{R}(w)$. Then we denote $w=vs$ which satisfies that  $\ell(w)= \ell(v)+1$. Note that ${\bf U}_w=({\bf U}_v)^{s} {\bf U}_s$. Thus
 for $x\in  U_{w, i+1}$, we can write $x=yz$,  where  $y\in ({\bf U}_v)^{s} $ and $z \in {\bf U}_s$.  Now assume that $z\ne e$.
 By Proposition \ref{suwformula}(i), if $sw_0w_K \leq w_0w_K$, then   we have
 $$ \dot{w}x\dot{w_0}  C(\mu)_K= a \dot{v} y^s f(z) \dot{w_0} C(\mu)_K$$
 for some $a\in \Bbbk$ and $f(z) \in  {\bf U}_s$.  If $sw_0w_K \geq w_0w_K$, then we have
 $$ \dot{w}x\dot{w_0}  C(\mu)_K= b\dot{v}y^s (f(z) -1) \dot{w_0} C(\mu)_K$$
 for some $b\in \Bbbk$ and $f(z) \in  {\bf U}^*_s$ by Proposition \ref{suwformula}(ii). In both cases, we get
 $$ |U_{w,i+1}\backslash\Omega'_w | \geq  (\widetilde{q}-1)(|U_{v,i+1}|-2|\Omega'_v|),$$
 which implies that
 $$\widetilde{q}^{\ell(w)}- |\Omega'_w| \geq (\widetilde{q}-1)(\widetilde{q}^{\ell(v)}- 2|\Omega'_v|).$$
By the inductive hypothesis $|\Omega'_v| \leq (3\widetilde{q})^{\ell(v)-1}$,  we get   $ |\Omega'_w|  \leq (3\widetilde{q})^{\ell(w)-1}$. Combining this and \eqref{omegaw}, we get $|\Omega_w|\geq\widetilde{q}^{n-1} ( \widetilde{q}-3^{\ell(w)-1})$.
\end{proof}

\begin{lemma}\label{Omegageq}
One has that $|\Omega| \geq q^{(i+1)!(n-1)}( q^{ i! (i-r-n+1)}-P)$, and in particular, $\Omega$ is nonempty when $i$ is large enough, where $P=\displaystyle \sum_{w\in W} 3^{\ell(w)-1}$.
\end{lemma}
\begin{proof}
Recall  the notation $H_u=  \bigcup_{t\in T_i}U_i t u t^{-1}$, where $u\in U_{i+1}$.
For any $x, y\in U_{i+1}$, we claim that $H_x=H_y$ or $H_x \cap H_y= \emptyset$.  Indeed, if $H_x \cap H_y\ne \emptyset$, then we have
$u_1 t_1xt^{-1}_1=u_2 t_2yt^{-1}_2$,  where  $u_1,u_2\in U_i$ and $ t_1,t_2\in T_i.$ Thus we get
$$ y= t^{-1}_2u^{-1}_2 u_1t_1xt^{-1}_1t_2= (t^{-1}_2u^{-1}_2 u_1t_2)   (t^{-1}_2t_1xt^{-1}_1t_2) \in H_x. $$
Thus we have  $H_y\subset H_x$ by easy calculation. By the same discussion,  we also have $H_x\subset H_y$, which implies that $H_x=H_y$.
Set $\mathscr{H}=\{H_u\mid u\in U_{i+1}\}$. Then $\mathscr{H}$ gives a partition of  $U_{i+1}$.
Noting that $|H_u|  \leq q^{i!(n+r)}$ for any $u\in U_{i+1}$, thus  we get
\begin{equation}\label{mathscrH}
|\mathscr{H}|\geq \frac{q^{(i+1)!n}}{q^{i!(n+r)}}=q^{i! (ni-r)}.
\end{equation}
Moreover, the Pigeonhole Principle implies that the cardinality of $H_u$'s satisfying $H_u\subset\bigcap_{w\in W}\Omega_w$ is at least $|\mathscr{H}|+|\bigcap_{w\in W}\Omega_w|-|U_{i+1}|$, and hence
\begin{align*}|\Omega| &\  \geq |\mathscr{H}|-\sum_{w\in W} |\Omega^c_w| \geq q^{i! (ni-r)} - \sum_{w\in W} |U'_{w,i+1}|  |\Omega'_w| \\
&\  \geq  q^{i!(ni-r)}- P q^{(i+1)!(n-1)},
\end{align*}
by Lemma \ref{owow'} and \eqref{mathscrH}.
Thus when $i$ is large enough,  we have
$$|\Omega| \geq q^{(i+1)!(n-1)}( q^{ i! (i-r-n+1)}-P)>0,$$
and the degree of the right side as a polynomial of $q$ is $i!(ni-r)$.
\end{proof}

\begin{lemma}\label{Omega-Gamma}
We let $$\Gamma=\{u\in\Omega  \mid \dot{w_0}^{-1} u^{-1} \big  (G_i\backslash Z(G_i) \big) u\dot{w_0} \cap B_{i+1} \ne \emptyset \}.$$
When $i$ is large enough, one has  that $\Omega\backslash\Gamma$ is nonempty. In particular, there exists $u\in \Omega$  such that
$$\dot{w_0}^{-1} u^{-1} (G_i\backslash Z(G_i)) u\dot{w_0} \cap B_{i+1} =\emptyset. $$
\end{lemma}
\begin{proof}
For each $w\in W$, we set
$$\Gamma_w=\{u\in\Omega  \mid \dot{w_0}^{-1} u^{-1} \big (B_iwB_i  \cap (G_i\backslash Z(G_i)) \big) u\dot{w_0} \cap B_{i+1} \ne \emptyset \}.$$
Thus we have $\Gamma=\bigcup_{w\in W} \Gamma_w$. We will estimate $|\Gamma_w|$ for any $w\in W$.
For any simple root $\alpha$, denote  $\pi_\alpha: {\bf U}\rightarrow {\bf U}_{\alpha}$ the projection.  For  $u\in \Omega$, then we get $\pi_\alpha(u) \notin U_{\alpha, i}$ for any simple root $\alpha$.  We claim that for any $u\in \Omega$,
$$\dot{w_0} ^{-1} u^{-1} \big (B_i \cap (G_i\backslash Z(G_i)) \big) u\dot{w_0} \cap B_{i+1}=\emptyset.$$
Indeed, for $x\in U_i$ and $t\in T_i$ such that $x\ne e$ and $t\notin Z(G_i)$, we write
$$ \dot{w_0} ^{-1} u^{-1} xt u\dot{w_0}=\dot{w_0} ^{-1} (u^{-1} x(t ut^{-1})) w_0 (\dot{w_0} ^{-1}tw_0).$$
Note that   $u^{-1} x(t ut^{-1}) \ne e$ for any $u\in \Omega$  when $x\ne e$ and $t\notin Z(G_i)$.   Therefore  we see that  $$\dot{w_0} ^{-1} (u^{-1} x(t ut^{-1})) w_0\notin B_{i+1}$$ and the
claim is proved. In particular,  we get $\Gamma_e=\emptyset$.

For $w\in W$ and $w\ne e$,  we denote
$$\widetilde{\Gamma}_w= \{x\in \Omega \mid  \dot{w_0} ^{-1}x^{-1} \dot{w} B_i x \dot{w_0} \cap  B_{i+1} \ne \emptyset \}.$$
Then it is easy to see that $|{\Gamma}_w | \leq |\widetilde{\Gamma}_w|$. Since  $x\in \Omega$,  we can  write
$$\dot{w_0} ^{-1}x^{-1} \dot{w}= \rho_w(x) \dot{w_0} \sigma_w(x),  \quad  \rho_w(x) \in B_{i+1}, \sigma_w(x) \in U_{i+1}.$$
It is easy to see that $x\in \widetilde{\Gamma}_w$ if and only if $ \sigma_w(x)  B_ix \cap T_{i+1} \ne \emptyset$.
This is equivalent to $ x T_{i+1}  \sigma_w(x)  \cap B_i \ne \emptyset$. When $t\in T_{i+1}\backslash T_i$, we have
$$  x t  \sigma_w(x) =t(t^{-1}xt ) \sigma_w(x) \notin B_i.$$
Therefore $x\in \widetilde{\Gamma}_w$ if and only if $ x T_{i}  \sigma_w(x)  \cap B_i \ne \emptyset$.
Thus we have
$$\widetilde{\Gamma}_w=\{x\in  \Omega \mid t^{-1}xt  \sigma_w(x)  \in U_i \ \text{for some} \ t\in T_i  \}.$$
Now let $\gamma_1,\gamma_2, \dots, \gamma_n $ be all the positive roots in $\Phi^+$  and we write $x= \varepsilon_{\gamma_1}(z_1)\varepsilon_{\gamma_2}(z_2) \cdots \varepsilon_{\gamma_n}(z_n)$.
Then we write
$$ \sigma_w(x) = \varepsilon_{\gamma_1}(f_1(z_1,\dots z_n))\varepsilon_{\gamma_2}(f_2(z_1,\dots z_n))\cdots \varepsilon_{\gamma_n}(f_n(z_1,\dots z_n)).$$
Using the structure theory  of algebraic groups, if we regard $z_1, z_2, \dots, z_n$ as the indeterminates, then $f_1, f_2, \dots, f_n$ are rational functions.
Noting that $w\ne e$, then we see that  not all the functions $f_1, f_2, \dots f_n$ are polynomial.
For any fixed $t_0\in T_i$, we write
$$t^{-1}_0xt_0 \sigma_w(x)= \varepsilon_{\gamma_1}(g_1(z_1,\dots z_n))\varepsilon_{\gamma_2}(g_2(z_1,\cdots z_n))\cdots \varepsilon_{\gamma_n}(g_n(z_1,\dots z_n)),$$
where $g_1, g_2, \dots, g_n$ are rational functions. We also see that not all $g_1, g_2, \dots, g_n$ are polynomials.
In particular, at least one function (denote by $g$)  is not the constant function.  For any fixed $u\in U_i$, we let
$$\mathcal{S}_{u}= \{(z_1,z_2,\dots, z_n) \in \mathbb{F}^n_{q^{(i+1)!}} \mid g(z_1,z_2, \dots, z_n)=u \}.$$
It is not difficult to see  that $|\mathcal{S}_{u}| \leq k q^{(n-1)\cdot (i+1)!}$  for some integer $k$ (determined by $g$).
Thus we have
$$|\widetilde{\Gamma}_w | \leq | \bigcup_{u\in U_i}\mathcal{S}_{u}| \leq \sum_{u\in U_i}  |\mathcal{S}_{u}| \leq k q^{(n-1)\cdot (i+1)!+n \cdot i!} $$

Note that  $\Gamma=\bigcup_{w\in W} \Gamma_w$.
Then there exists an integer $N$ such that
$$| \Gamma| \leq \sum_{w\in W} |\Gamma_w |\leq  N |W|  q^{(n-1)\cdot (i+1)!+n \cdot i!}. $$
When $i$ is large enough, the degree of the right side as a polynomial of $q$ is  $(n-1)\cdot (i+1)!+n \cdot i!$.
By Lemma \ref{Omegageq}, we have
$$|\Omega| \geq q^{(i+1)!(n-1)}( q^{ i! (i-r-n+1)}-P),$$
and the degree of the right side as a polynomial of $q$ is $i!(ni-r)$. We have
$$ i!(ni-r)-  ((n-1)\cdot (i+1)!+n \cdot i!)= (i+1)!- i! (2n+r)>0$$
when $i$ is large enough. Therefore $ \Omega\backslash\Gamma$ is nonempty.
\end{proof}

\begin{proof}[Proof of Proposition \ref{keylemma}]
It suffices to show the following  Claim $(\clubsuit)$:  There exists $u\in \Omega$ such that  for any $g \in G_i\backslash Z(G_i)$ and $w\in W$,
 $g u \dot{w} \notin u\dot{w_0} B_{i+1}$. 
 Once the Claim $(\clubsuit)$ is proven, then for such $u \in \Omega$, we have $g u \dot{w} {\bf 1}_{\mu} \notin \Bbbk u w_0 {\bf 1}_{\mu}$ and hence
$gu\dot{w_0} C(\mu)_K \notin \Bbbk u\dot{w_0} C(\mu)_K $ for any $g \in G_i\backslash Z(G_i)$.
Thus it is not difficult to see that  the coefficients of  $u \dot{w_0} C(\mu)_K$ in the expression of
  $$\xi_{i, u}= \sum_{z\in U_{i+1} } b_z z\dot{w_0} C(\mu)_K,  \quad b_z\in \Bbbk.  $$
 is nonzero, which implies that  $\xi_{i, u} \in E_{i+1}(\mu)_K$ is nonzero.

 Now we prove the Claim $(\clubsuit)$. Firstly we show that if $w\ne w_0$ and $u\in \Omega$, then $g u \dot{w} \notin u \dot{w_0} B_{i+1}$. Indeed, we write $g=x\dot{v}y$, where $x\in U_{i}$, $v\in W$ and $y\in B_i$.  We determine whether the element $ \dot{ w_0} u^{-1} g u \dot{w}= \dot{w_0} u^{-1} x\dot{v}y u \dot{w }$ is in $B_{i+1}$.
By the property of $\Omega$, it is easy to see that $ \dot{w_0} u^{-1} x \dot{v} \in B_{i+1} w_0 B_{i+1}$ for any $u\in \Omega$. Thus when $w\ne w_0$, we get $\dot{ w_0} u^{-1} g u \dot{w} \notin B_{i+1}$, which proves the  Claim $(\clubsuit)$ for $w\ne w_0$. For the case $w=w_0$,  the Claim $(\clubsuit)$ follows immediately from Lemma \ref{Omega-Gamma}. Therefore Proposition \ref{keylemma} is proved. 
\end{proof}

Now for each integer $i$, we fix $u_i\in U_{i+1}$ such that $u_i \in \Omega\backslash\Gamma$  using the same notation in Proposition \ref{keylemma}. Thus
$\xi_i:= \xi_{i,u_i}  \in E_{i+1}(\mu)_K $ is nonzero. For convenience, we set
$$ \eta_{i}= \sum_{p\in P_{J', i}} \lambda(p)^{-1}p u_i  w_0 C(\mu)_K.$$
With this notation, we see that  $\xi_{i}=\displaystyle  (\sum_{w\in W_J} (-1)^{\ell(w)}w)  \eta_{i}$.
 For any $x\in  P_{J', i}$, it is easy to see that $ x\eta_{i} =\lambda(x)\eta_{i} .$ So we get a natrual  $\Bbbk G_i$-module homomorphism
 $$\varphi_i:  \op{Ind}^{G_i}_{P_{J', i}} {\bf 1}_{\lambda} \rightarrow \Bbbk G_i \eta_{i}, \quad  {\bf 1}_{\lambda} \mapsto   \eta_{i},$$
which induces a $\Bbbk G_i$-module  homomorphism
$$\varphi_i: E_i(\lambda)_{J} \rightarrow E_{i+1}{(\mu)_K},  \quad C(\lambda)_J \rightarrow \xi_i. $$
By the choice of $u_i$,  we see that $\varphi_i$ is nonzero and injective.
Therefore we can define a homomorphism  $f_i\in \text{Hom}_{G_i}(M_i, M_{i+1})$ by
$$f_i(C(\lambda)_J ) = C(\lambda)_J+ \xi_i   \quad \text{and} \quad  f_i(m)=m,  \forall  m\in E_i(\mu)_K.$$
For $i<j$, we denote by $f_{i,j}:~M_i\rightarrow M_j$ the composition $f_{j-1}\circ\cdots\circ f_{i+1}\circ f_i$ which is clearly in $\op{Hom}_{G_i}(M_i,M_j)$. Then $\{M_i,f_{i,j}\}$ is a direct system of vector spaces.   It is easy to see that each  $f_i$  is injective and thus all $f_{i,j}$ are injective. Let $M$ be the direct limit of $\{M_i,f_{i,j}\}$, which is a $\Bbbk {\bf G}$-module.

Suppose that $m,m'\in M$ are represented by $m_i\in M_i$ and $m'_j\in M_j$ respectively. Then $m=m'$ if and only if $f_{i,k}(m_i)=f_{j,k}(m'_j)$ for some $k$. By the injectivity mentioned above, this is equivalent to
 $$
 \begin{cases}
 f_{i,j}(m_i)=m_j' &\ {\rm if}~i<j;\\
 m_i=m_i'          &\ {\rm if}~i=j;\\
 f_{j,i}(m_j')=m_i &\ {\rm if}~i>j.
 \end{cases}
 $$
The space $M$ is a $\Bbbk {\bf G}$-module via the following way. For any $g\in {\bf G}$ and $m \in M$, suppose that  $m$ is represented by $m_i\in M_i$. We choose $k$ large enough such that $k\ge i$ and $g\in G_k$. Define $gm\in M$ to be the element represented by $gf_{ik}(m_i)\in M_k$.

\begin{lemma}
Let $M$ be the  $\Bbbk {\bf G}$-module constructed before. Then we have a short exact sequence
$$0 \rightarrow E(\mu)_K  \rightarrow M \rightarrow E(\lambda)_J \rightarrow 0$$
of $\Bbbk {\bf G}$-modules.
\end{lemma}

\begin{proof}
Let $N$ be the subspace of $M$ such that each element $m\in N$ is represented by $m_i\in  E_i(\mu)_K$ for some integer $i$.
Since the restriction of $f_i$ on $E_i(\mu)_K$ is simply the usual inclusion $E_i(\mu)_K \rightarrow E_{i+1}(\mu)_K$, we have $N \cong E(\mu)_K$ as $\Bbbk {\bf G}$-modules. In the following, we show that $M/N \cong E(\lambda)_J$ as $\Bbbk {\bf G}$-modules.

Let $0\ne \overline{m} \in M/N$ and we choose one representative $m\in M$. Then there is an integer $j$ such that $m$ is represented by $m_j \in M_j$. Noting that $M_j= E_j(\lambda)_J \oplus  E_j(\mu)_K$, we write $m_j= x_j +y_j$, where $x_j\in E_j(\lambda)_J, y_j \in E_j(\mu)_K. $
Let $x\in M$ be the equivalence class containing $x_j$.  Then we get  $\overline{m}=\overline{x}$.  Thus $M/N \cong E(\lambda)_J$ as $\Bbbk$-vector spaces.

On the other hand, we consider $g \overline{m}$ for any $g\in G$ and $m\in M$.  There exists an integer $k$ such that $g\in G_k$ and $m$ is represented by $m_k \in M_k$. Write $m_k= x_k +y_k$,  where $x_k\in E_k(\lambda)_J, y_k \in E_k(\mu)_K$. Thus $g m_k = gx_k +g y_k$ and
 $gx_k\in E_k(\lambda)_J$, $gy_k \in E_k(\mu)_K$. Let $x\in M$ be the equivalence class containing $x_k$. Thus we have
$$ g \overline{m}=\overline{gm} =\overline{gx} =  g \overline{x},$$
which implies that $M/N \cong E(\lambda)_J$ as $\Bbbk {\bf G}$-modules. The lemma is proved.

\end{proof}

\begin{proposition}
The short exact sequence
$$0 \rightarrow E(\mu)_K  \rightarrow M \rightarrow E(\lambda)_J \rightarrow 0$$
constructed  before is non-split.
\end{proposition}

\begin{proof}
  It is enough to show that there is no  element $m\in M- E(\mu)_K$ such that $m\in M^{{\bf U}'_{w_J} }$.
Suppose that such $m\in M- E(\mu)_K$ exists and assume that $m$ is represented by $m_i\in M_i$.
Since $M_i= E_i(\lambda)_J \oplus  E_i(\mu)_K$, we write
$$m_i=\sum_{w\in W,  x\in U_i} a_{w,x} x w C(\lambda)_J + \sum_{v\in W, y\in U_i} b_{v,y} y v C(\mu)_K, $$
where not all $a_{w,x}$ are zero since $m\in M\backslash E(\mu)_K$. Thus we have
$$f_i (m_i)=\sum_{w\in W,  x\in U_i} a_{w,x} x w (C(\lambda)_J+ \xi_{i}) + \sum_{v\in W, y\in U_i} b_{v,y} y v C(\mu)_K ,$$
which is in $M^{U'_{w_J, {i+1}}}_{i+1}$ since $m\in M^{{\bf U}'_{w_J} }$. By Proposition \ref{keylemma}, we see that
$$0\ne \displaystyle \sum_{w\in W,  x\in U_i} a_{w,x} x w \xi_i \in E_{i+1}(\mu)_K\backslash E_i(\mu)_K$$
if $i$ is large enough. Note that $|G_{i}| < |U'_{w_J, {i+1}}|$ when $i$ is large enough.  Thus we see that $f_i (m_i) \notin M^{U'_{w_J, {i+1}}}_{i+1}$,  which is a contradiction.  The proposition is proved.
\end{proof}

In the paper  \cite{D1}, we introduce the principal representation category $\mathscr{O}({\bf G})$. The category $\mathscr{O}({\bf G})$  is  defined to be the full subcategory of $\Bbbk{\bf G}$-Mod such that any object $M$ in $\mathscr{O}({\bf G})$ is of finite length and its composition factors are $E(\theta)_J$ for some $\theta \in \widehat{\bf T}$ and $J\subset I(\theta)$. Let  $\widehat{\bf C}$ be the central characters of ${\bf G}$, which consists of  the homomorphisms $\chi: {\bf C} \rightarrow \Bbbk^*$. For any central character $\chi\in \widehat{\bf C}$, we let $\mathscr{O}({\bf G})_{\chi}$ the subcategory of $\mathscr{O}({\bf G})$ such that the composition factors of each object in $\mathscr{O}({\bf G})_{\chi}$ are $E(\lambda)_J$, where  $\lambda|_ {\bf C} = \chi $. Thus $\mathscr{O}({\bf G})_{\chi}$ is a block and we get the block decomposition of  $\mathscr{O}({\bf G})$ as follows:
$$\mathscr{O}({\bf G})= \bigoplus_{\chi \in \widehat{\bf C}} \mathscr{O}({\bf G})_{\chi} $$
by Theorem \ref{ifpart} and Theorem \ref{onlyifpart}.

\bigskip

\noindent{\bf Statements and Declarations}  The authors  declare that they have no conflict of interests with others.

\medskip

\noindent{\bf Data Availability}  Data sharing not applicable to this article as no datasets were generated or analysed during the current study.

\medskip

\noindent{\bf Acknowledgements} The authors are grateful to  Nanhua Xi for his inspiring encouragements and helpful comments. This work is sponsored by  NSFC-12101405.

\bigskip

\bibliographystyle{amsplain}

\end{document}